\UseRawInputEncoding
\documentclass[12pt, reqno]{amsart}
\usepackage{amssymb}
\usepackage{amsfonts}
\usepackage{amssymb,amsmath,amscd}
\usepackage[colorlinks=true]{hyperref}
\usepackage{mathrsfs}
\usepackage{amsthm}
\newtheorem{Theorem}{\indent Theorem}[section]

\newtheorem{Lemma}[Theorem]{\indent Lemma}
\newtheorem{Corollary}[Theorem]{\indent Corollary}

\theoremstyle{remark}

\begin{document}
\centerline{
\bf On general divisor functions over  Piatetski-Shapiro sequences}
\bigskip
\centerline{Wei Zhang}

\centerline{
School of Mathematics and Statistics,
Henan University,
Kaifeng  475004, Henan, China}
\centerline{zhangweimath@126.com}
\bigskip

\textbf{Abstract}
In this paper, we consider the general divisor functions over Piatetski-Shapiro sequences. We can give some general results which contain some special divisor functions.
Precisely, we  extend the divisor problem over Piatetski-Shapiro sequences to the function $f(n),$ where
$f(n)\ll n^{\varepsilon},$ $$f(n)=\sum_{n=n_{1}n_{2}}
\tau(n_{1})g(n_{2}),$$
$\tau(n)$
is the number of representations of $n$ as product of two natural
numbers
 and
\[
\sum_{1\leq n\leq x}|g(n)|\ll x^{5/8+\varepsilon}.
\]
On the other hand, we also considered these arithmetic functions over Piatetski-Shapiro sequences in arithmetic progressions.



\bigskip
\numberwithin{equation}{section}

\section{\bf{Introduction}}
Piatetski-Shapiro sequences are named in honor of Piatetski-Shapiro, who proved that for any number $c\in(1,12/11)$ there are infinitely many primes of the form $[n^{c}]$ by showing that
 \begin{align}\label{01}
 \sum_{\substack{1\leq n\leq N\\ [n^{c}]\ \textup{is}\ \textup{prime}}}1
 =(1+o(1))\frac{N}{c\log N}.
 \end{align}
 The admissible range for $c$ in this problem has  been extended by many experts over the years.
And to date, the largest admissible $c$-range for (\ref{01}) seems to be $c\in(1,2817/2425)$ due to Rivat and Sargos  \cite{RS0} (see also the
references to the previous record holders they gave in their paper), which improves previous smaller range of Rivat \cite{RD} and  Liu-Rivat \cite{LR}. Naturally, also
lower bound sieves have been employed, and the corresponding current record is a version of (\ref{01}) with a lower bound of the right order of magnitude instead of an
asymptotic formula and $c\in(1,243/205)$ due to Rivat and Wu \cite{RW}.
Moreover, this type of prime number
 theorem has been generalized from many different perspectives and dimensions.
Many experts considered the  Piatetski-Shapiro primes in arithmetic progressions.
For example, in \cite{LW}, Leitmann and Wolke showed that for $c\in(1,12/11),$ $q\in\mathbb{N},$ $(q,a)=1,$ one has
 \begin{align*}
 \sum_{\substack{1\leq n\leq x\\ [n^{c}]\ \textup{is}\ \textup{prime}\\
 [n^{c}]\equiv a(\textup{mod}\ q)}}1
 =(1+o(1))\frac{N}{c\varphi(q)\log N}.
 \end{align*}
Similar results were also given in \cite{BBB,GLZ}.
Other special extension of the Piatetski-Shapiro prime number theorem can be seen in \cite{B1,D1,D2,D4}.
And the generalization of the Piatetski-Shapiro prime number theorem  related to the Hecke eigenvalues can be seen  in \cite{B2,B3,B4}.
Many people have also considered other topics related to the Piatetski-Shapiro sequence. For example, in \cite{D3}, a Bombieri-Vinogradov-type theorem related to the Piatetski-Shapiro sequence was considered. In \cite{MM}, the Roth type theorem related to the Piatetski-Shapiro primes was considered.
On the other hand, many experts also considered the divisor functions over Piatetski-Shapiro sequences. Let $\tau_{k}(n)$
be the number of representations of $n$ as product of $k$ natural
numbers. For example, in 1997, Soliba \cite{S97} showed that
for $c\in(1,9/8),$ one has
\[
\sum_{1\leq n\leq x}\tau_{3}([n^{c}])
=xQ_{2}(\log x)+O\left(x/\log x\right),
\]
where $Q_{2}(x)$
is a polynomial of degree $2.$
The above result was generalized and improved by    Arkhipov, Soliba and  Chubarikov \cite{ASC}. Their classical result  is that for any fixed $k\geq 2,$ $c\in(1,8/7),$ one has
\[
\sum_{1\leq n\leq x}\tau_{k}([n^{c}])
=xQ_{k-1}(\log x)+O\left(x/\log x\right),
\]
where $Q_{k-1}(x)$
is a polynomial of degree $k-1.$
These results were subsequently improved in  \cite{AC,LZ}.
The result in   \cite{AC}  states that for any fixed $k\geq 2,$ $c\in(1,149/129),$ one has
\[
\sum_{1\leq n\leq x}\tau_{k}([n^{c}])
=xQ_{k-1}(\log x)+O\left(x/\log x\right),
\]
where $Q_{k-1}(x)$
is a polynomial of degree $k-1,$
which has the same range of $c\in(1,149/129)$ as their result related to the distribution of prime numbers in Piatetski-Shapiro sequences. Recently, by using a deep result of Jutila \cite{Ju}, for the case $k=2$ (i.e. $\tau_{2}(n)=\tau(n)$), in \cite{WZ}, for $c\in(1,6/5),$ it is proved that
\begin{align*}
 \sum_{1\leq n\leq x}\tau\left(\left[n^{c}\right]
\right)=
cx\log x+(2\gamma_{0}-c)x
+O\left(x/(\log x)\right),
 \end{align*}
where  $\gamma_{0}$ is the Euler's constant.
However, only relying  on using the deep result of Jutila \cite{Ju} to deal with the key exponential sums, it is not easy for one to deal with the general divisor functions.
When dealing with   general divisor functions, as usual, in \cite{ASC,AC,LZ}, three-dimensional exponential sums were used. However, it is not easy to obtain the desired range $c\in(1,6/5)$ only relying  on using the three dimensional exponential sums.
In this paper, we consider some general divisor functions over the Piatetski-Shapiro sequences, which covers many results. Our first result can be stated as the following theorem which  can be  obtained  by using three dimensional exponential sums with a desired range $c\in(1,6/5)$.
Throughout   this short note, $\varepsilon$ is a very small positive real number and varies depending on its specific location.
\begin{Theorem}\label{th3}
Let $f$ be a   real-valued arithmetic function such that
$f(n)\ll n^{\varepsilon},$   $$f(n)=\sum_{n_{1}n_{2}=n}
\tau(n_{1})g(n_{2})$$
and
\[
\sum_{1\leq n\leq x}|g(n)|\ll x^{5/8+\varepsilon}.
\]
Then for any sufficiently large $x,$  fixed $c\in(1,6/5)$, $\gamma=1/c,$ we have
\begin{align*}
 \sum_{1\leq n\leq x}f\left(\left[n^{c}\right]
\right)=
\int_{1}^{x^{c}} \gamma u^{\gamma-1}
 d\left(\sum_{1\leq n\leq u}f(n)\right)
+O\left(x^{1-\varepsilon}\right),
\end{align*}
where $\varepsilon$ is an arbitrary small positive constant and the implied constant may depend on $\varepsilon,f,c.$
\end{Theorem}
When dealing with the average number of direct factors of a finite Abelian group, one usually needs to consider the sum
$$\tau_{(1,1,2,2)}(n): =\sum_{n_{1}n_{2}n_{3}^{2}
n_{4}^{2}=n}1.$$
We also have
\begin{align*}
\tau_{(1,1,2,2)}(n)
=\sum_{n=n_{1}n_{2}}\tau(n_{1})
g(n_{2}),\end{align*}
where
\begin{align*}
g(n)=
\begin{cases}
\tau(n^{1/2})\ &\textup{for}\ n\ \textup{is\ a\ square},\\
0\ &\textup{otherwise}.
\end{cases}
\end{align*}
Hence we have
$$\sum_{1\leq n\leq x}g(n)\ll x^{1/2+\varepsilon}.$$
In Lemma 2 of \cite{Kr}, it is shown that
\begin{align}\label{k1}
\sum_{1\leq n\leq x}\tau_{(1,1,2,2)}(n)=B_{1}x\log x+B_{2}x+O(x^{1/2}\log x),
\end{align}
where $B_{1}$, $B_{2}$ are certain constants. We have
\begin{align}\label{hy}
\begin{split}
&\int_{1}^{x^{c}}\gamma u^{\gamma-1}d\left(B_{1}u\log u+B_{2}u\right)\\
&=\int_{1}^{x^{c}}\gamma u^{\gamma-1}
B_{1}\log udu+(B_{1}+B_{2})x+O(1)\\
&=\int_{1}^{x^{c}}
B_{1}\log udu^{\gamma}+(B_{1}+B_{2})x+O(1)\\
&=cB_{1}x\log x-cB_{1}x+(B_{1}+B_{2})x+O(1).
\end{split}
\end{align}
Hence, for this arithmetic function, by Theorem \ref{th3},
and partial summation,
we have the following corollary.
\begin{Corollary}
Let
$$\tau_{(1,1,2,2)}(n): =\sum_{n_{1}n_{2}n_{3}^{2}
n_{4}^{2}=n}1.$$
For any sufficiently large $x,$ fixed $c\in(1,6/5)$ and $\gamma=1/c,$ we have
\[
\sum_{1\leq n\leq x}\tau_{(1,1,2,2)}([n^{c}])
=cB_{1}x\log x+(B_{1}+B_{2}-cB_{1})x
+O(x^{1-\varepsilon}),
\]
where $B_{1}, B_{2}$ are certain constants given by (\ref{k1}), $\varepsilon$ is an arbitrary small positive constant and the implied constant may depend on $\varepsilon,c.$
\end{Corollary}
A positive
 number $q$ is called
$k$-free integer  if and only if
$
m^{k}|q\Longrightarrow m=1.
$
For sufficiently large $x\geq1,$ it
is well known that
\[
\sum_{n\in\mathcal{Q}_{k}}n^{-s}=
\frac{\zeta(s)}{\zeta(ks)},\ \ \ \Re s>1
\]
and
\begin{align}\label{kf}
\sum_{n\leq x,\ n\in\mathcal{Q}_{k}}1=
\frac{x}{\zeta(k)}+O(x^{1/k}),
\end{align}
where $\mathcal{Q}_{k}$ is
the set of positive $k$-free integers.
The above theorem can also be used to give some information related to  $k$-free divisor function $\tau_{(k)}(n)$ over the Piatetski-Shapiro sequences, where
$$\tau_{(k)}(n)=\sum_{\substack{d|n\\
d\in\mathcal{Q}_{k}}}1.$$
We also have
\begin{align*}
\tau_{(k)}(n)
=\sum_{n=n_{1}n_{2}}\tau(n_{1})
g(n_{2}),\end{align*}
where
\begin{align*}
g(n)=
\begin{cases}
\mu(n)\ &\textup{for}\ n\ \textup{is\ a\ $k$th\ \textup{power}},\\
0\ &\textup{otherwise}.
\end{cases}
\end{align*}
Hence we have
$$\sum_{1\leq n\leq x}g(n)\ll x^{1/k+\varepsilon}.$$
In fact, by using the main result of \cite{P}, similar arguments as (\ref{hy}), we can give an asymptotic formula of the following form.
\begin{Corollary}
Let $\tau_{(k)}(n)$ be the $k$-free divisor function.
Then for any sufficiently large $x,$   fixed  $c\in(1,6/5)$ and $\gamma=1/c,$ we have
\begin{align*}
 \sum_{1\leq n\leq x}\tau_{(k)}\left(\left[n^{c}\right]
\right)=
cx\log x +\left(\frac{2\gamma_{0}-1}{\zeta(k)}
-\frac{k\zeta'(k)}{\zeta^{2}(k)}+1-c\right)x+
O\left(x^{1-\varepsilon}\right),
 \end{align*}
 where $\gamma_{0}
 $
 is the Euler's constant,  $\zeta(s)$ is the well-known Riemann zeta function, $\varepsilon$ is an arbitrary small positive constant and the implied constant may depend on $\varepsilon,c.$
\end{Corollary}
On the other hand, for $\gamma_{0}$ being the Euler's constant, by the well-known formula
\[
\sum_{1\leq n\leq x}\tau(n)=x\log x+(2\gamma_{0}-1)x+O(x^{1/2}),
\]
Theorem \ref{th3}, one can obtain the following corollary, which gives improvement for the result of Arkhipov-Chubarikov  \cite{AC} and was also given by a  different idea  (using a deep result of Jutila \cite{Ju}) in \cite{WZ}.
\begin{Corollary}
For   any sufficiently large $x,$ fixed   $c\in(1,6/5)$ and $\gamma=1/c,$ we have
\begin{align*}
 \sum_{1\leq n\leq x}\tau\left(\left[n^{c}\right]
\right)=
cx\log x+(2\gamma_{0}-c)x
+O\left(x^{1-\varepsilon}\right),
 \end{align*}
where the implied constant may depend on $\varepsilon,c.$
\end{Corollary}
\bigskip
Many experts also considered the  Piatetski-Shapiro primes in arithmetic progressions. For example,   the method of Leitmann and Wolke \cite{LW} implies  that for $c\in(1,12/11),$ $d\in\mathbb{N},$ $(d,a)=1,$ one has
\begin{align*}
\sum_{\substack{1\leq n\leq N\\ [n^{c}]\ \textup{is}\ \textup{prime}\\
 [n^{c}]\equiv a(\textup{mod}\ d)}}1
 =(1+o(1))\frac{N}{c\varphi(d)\log N},
 \end{align*}
where $\varphi(n)$ is the Euler totient function.
Next, we also considered the above  special  arithmetic functions over Piatetski-Shapiro sequences in arithmetic progressions, which gives further generalization.
The following result has relation to the Fourier coefficients of Hecke-Maass cusp forms. Now we introduce some related backgrounds. One can refer to \cite{IK} for details of these interesting results. Let $S_{r}$ be the space of
Hecke-Maass cusp forms of Laplace eigenvalue $\lambda=1/4+r^{2}$ with respect to the full
modular group $SL_{2}(\mathbb{Z})$. At the cusp at infinity, for $\nu\in S_{r}$, we have the Fourier expansion
\[
\nu(z)=\sum_{n\neq 0}\lambda_{\nu}(n)\sqrt{y}K_{ir}(2\pi ny)e(nx),
\]
where $e(x):=e^{2\pi ix},$ $\lambda_{\nu}(1)=1,$ $\lambda_{\nu} (n)\ll n^{7/64+\varepsilon},$ $\lambda_{\nu}(n)\in\mathbb{R}$ denotes the $n$-th eigenvalue of Hecke operator $T_{n}$ and $K_{ir}$ is the $K$-Bessel function. And we also have
\begin{align}\label{Hecke}
\lambda_{\nu}(mn)=\sum_{d|(m,n)}\mu(d)
\lambda_{\nu}
\left(\frac{m}{d}\right)
\lambda_{\nu}
\left(\frac{n}{d}\right).
\end{align}
The $L$-function $L(s,\nu)$ is defined as
\[
L(s,\nu)=\sum_{n=1}^{\infty}
\frac{\lambda_{\nu}(n)
}{n^{s}}=\prod_{p}
\left(1-\lambda_{\nu}(p) p^{-s}+p^{-2s} \right)^{-1}, \ \ \Re(s)>1.
\]
Then we can state our next result as the following.
\begin{Theorem}\label{th5}
Let $f$ be  a   real-valued arithmetic function
$f(n)\ll n^{\varepsilon},$   $$f(n)=\sum_{n_{1}n_{2}=n}
\tau(n_{1})g(n_{2})$$
and
\[
\sum_{1\leq n\leq x}|g(n)|
\ll x^{\frac{5+4\alpha}{8+4\alpha}
+\varepsilon},
\]
where $\alpha$ is the constant for which $\lambda_{\nu}(n)\ll n^{\alpha}$ holds.
Then for any sufficiently large $x,$
fixed $c\in \left(1,\frac{8+4\alpha}
{7
+4\alpha}\right)$, $\gamma=1/c,$  there is a constant $\varepsilon>0$ such that for all integer $a,d$ with $\gcd(a,d)=1,$
\begin{align*}
 \sum_{\substack{1\leq n\leq x\\ [n^{c}]\equiv a(\mod d)}}f\left(\left[n^{c}\right]
\right)=
\int_{1}^{x^{c}} \gamma u^{\gamma-1}
 d\left(\sum_{\substack{1\leq n\leq u\\ n\equiv a(\mod d)}}f(n)\right)
+O\left(x^{1-\varepsilon}\right),
\end{align*}
where  the implied constant may depend on $\varepsilon,f,c,a,d.$
\end{Theorem}
For the divisor function,
in unpublished works, it has been discovered independently
by Selberg and Hooley that for any $\varepsilon>0,$ there exists some $\delta>0$  such that for all integer $a,d$ with $\gcd(a,d)=1,$ we have the following result
\begin{align}\label{shpv}
 \sum_{\substack{1\leq n\leq x\\ n\equiv a(\mod d)}}\tau\left(n\right)
 -\frac{1}{\phi(d)}\sum_{1\leq n\leq x,\ (n,d)=1}\tau(n)\ll \frac{x^{1-\delta}}{d}
\end{align}
holds uniformly for $d\leq x^{2/3-\varepsilon}.$ This deep result can also be seen in \cite{PV}.
By Theorem \ref{th5}, for $\alpha=7/64,$ we can obtain the following corollary, which can be seen as an analogous result over Piatetski-Shapiro sequences of the result of Selberg and Hooley (\ref{shpv}).
\begin{Corollary}
For fixed $c\in(1,135/119),$ $\gamma=1/c,$   there is a constant $\varepsilon>0$ such that for all integer $a,d$ with $\gcd(a,d)=1,$
\begin{align*}
 \sum_{\substack{1\leq n\leq x\\ [n^{c}]\equiv a(\mod d)}}\tau\left(\left[n^{c}\right]
\right)=
\int_{1}^{x^{c}} \gamma u^{\gamma-1}
 d\left(\sum_{\substack{1\leq n\leq u\\ n\equiv a(\mod d)}}\tau(n)\right)
+O\left(x^{1-\varepsilon}\right),
 \end{align*}
where the implied constant may depend on $\varepsilon,c,a,d.$
\end{Corollary}

\section{Preliminary lemmas}
Before the proof of our theorems, we need to introduce some related lemmas.
Let $\psi(t)=t-[t]-1/2$ for $t\in\mathbb{R}$ and $\delta\geq0.$
We  need
the following  well-known lemma.
This result can be seen in Theorem A.6 in \cite{GK} or Theorem 18 in \cite{Va}.
\begin{Lemma}\label{z3}
For $0<|t|<1,$ let
$$W(t) = \pi t(1-|t|)\cot\pi t + |t|.$$
  For $x\in\mathbb{R},$ $H\geq1,$ we define
$$\psi^{*}(x)=-\sum_{1\leq |h|\leq H}(2\pi ih)^{-1}W\left(\frac{h}{H+1}\right)e(hx)$$
and
\[
\delta(x)=\frac{1}{2H+2}\sum_{|h|\leq H}\left(1-\frac{|h|}{H+1}\right)e(hx).
\]
Then for all real $x$, $\delta(x)$ is non-negative, and we have
$$|\psi^{*}(x)-\psi(x)|\leq
\delta(x).$$
\end{Lemma}
Next, we need the following well-known lemma in \cite{GK}.

\begin{Lemma}\label{z2}
Let $ f $ be two times continuously differentiable on a subinterval $ \mathcal{I} $ of $ (N,2N] $.
Suppose that for some $ \lambda > 0 $, the inequalities
$$
\lambda \ll f''(t) \ll \lambda \quad (t \in \mathcal{I})
$$
hold, where the implied constants are independent of $ f $ and $ \lambda $. Then
$$
\sum_{n \in \mathcal{I}} e(f(n)) \ll N\lambda^{1/2} + \lambda^{-1/2}.
$$
\end{Lemma}
 \begin{Lemma}\label{rs1}
For real numbers $\alpha_{1}, \alpha_{2}, \alpha_{3}$ such that $\alpha_{1}\alpha_{2}\alpha_{3}
(\alpha_{1}-1)\neq0$.
For $X>0,$ $M_{1}\geq1,$ $M_{2}\geq 1,$ and $M_{3}\geq 1,$ let
\[
S(M_{1},M_{2},M_{3}):=
\sum_{m_{2}\sim M_{2}}\sum_{m_{3}\sim M_{3}}
\left|\sum_{m_{1}\sim M_{1}}e\left(X\frac{m_{1}^{\alpha_{1}} m_{2}^{\alpha_{2}}
 m_{3}^{\alpha_{3}}}
{M_{1}^{\alpha_{1}}M_{2}^{\alpha_{2}}
M_{3}^{\alpha_{3}}}
\right)\right|^{*},
\]
where $e(t)=e^{2\pi i t}.$   For any $\varepsilon>0,$ we have
\begin{align*}
S(M_{1},M_{2},M_{3})(XM_{1}M_{2}M_{3})^{-\varepsilon}&\ll
 \left(X
M_{1}^{2}M_{2}^{3}
M_{3}^{3}\right)^{1/4} +M_{1}^{1/2} M_{2}M_{3}
 +X^{-1}M_{1}M_{2}M_{3},
\end{align*}
 where
 \[
 \left|\sum_{M<n\leq 2M}z_{m}\right|^{*}:=\max_{M<u\leq 2M}\left|\sum_{M<n\leq u}z_{m}\right|
 \]
and
the implied constant may depend on $\alpha_{1},\alpha_{2},\alpha_{3},$ and $\varepsilon.$
\end{Lemma}
\begin{proof}
This can be obtained by  Theorem 3 in \cite{RS}.
\end{proof}
The following two lemmas are Lemma 2.5 and Lemma 2.4 in \cite{GK}.
\begin{Lemma}\label{t1}
Suppose that $\rho(n)$ is a complex valued function such that $\rho(n)=0$ if $n\notin I.$ If $Q\geq 2$ is a positive integer, then we have
\[
\left|\sum_{n}\rho(n)\right|^{2}
\leq \frac{|I|+Q}{Q}
\sum_{|q|<Q}\left(1-\frac{|q|}{Q}\right)
\sum_{n}\rho(n)\overline{\rho(n+q)}.
\]
\end{Lemma}

\begin{Lemma}\label{t2}
Let $M>0,$ $N>0,$ $u_{m}>0,$ $v_{n}>0,$ $A_{m}>0,$ $B_{n}>0$ $(1\leq m\leq M,$ $1\leq n\leq N),$ and let $Q_{1}$ and $Q_{2}$ be given non-negative numbers, $Q_{1}\leq Q_{2}.$ Then there is a $Q$ such that
$Q_{1}\leq Q \leq Q_{2}$ and
\begin{align*}
\sum_{m=1}^{M}A_{m}Q^{u_{m}}
+\sum_{n=1}^{N}B_{n}Q^{-v_{n}}&
\ll \sum_{m=1}^{M}\sum_{n=1}^{N}
\left(A_{m}^{v_{n}}B_{n}^{u_{m}}
\right)^{1/(u_{m}+v_{n})}\\
&+ \sum_{m=1}^{M}A_{m}Q_{1}^{u_{m}}+
\sum_{n=1}^{N}B_{n}Q_{2}^{-v_{n}}.
\end{align*}
\end{Lemma}

We also need the following classical lemma to deal with certain exponential sums.
\begin{Lemma}
\label{t4}
Let $F(u)$   be real function  in $[a,b],$   and $F'(u)$ is monotone.
If $F'(u)\geq m>0$ or $F'(u)\leq -m<0,$ then
\[
\int_{a}^{b}e(F(u))du\ll \frac{1}{m}.
\]
\end{Lemma}
\begin{proof}
See page 71 of \cite{T}.
\end{proof}
We also need something related to the binary additive divisor problem which consists of the asymptotic evaluation
 of the sum
 \[
 D(x,q) :=\sum_{1\leq n \leq x}\tau(n)\tau(n+q),
 \]
 where $\tau(n)$ is  the number of positive divisors of $n$ and $q$ is a
 natural number, not necessarily fixed.
Let
\[
c_{1}(q)=\frac{\sigma_{-1}(q)}{\zeta(2)},
\]
\[
c_{2}(q)=c_{1}(q)\left(4\gamma_{0}-2
-4\frac{\zeta'(2)}{\zeta(2)}
-4\frac{\sigma'_{-1}(q)}
{\sigma_{-1}(q)}\right),
\]
\begin{align*}
c_{3}(q)&=c_{1}(q)\left(\left(
2\gamma_{0}-1-2\frac{\zeta'(2)}{\zeta(2)}
-\frac{\sigma'_{-1}(q)}
{\sigma_{-1}(q)}\right)^{2}
\right.\\
&\left.+
1-4\frac{\zeta''(2)}{\zeta(2)}
+
4\left(\frac{\zeta'(2)}
{\zeta(2)}\right)^{2}
+4\frac{\sigma''_{-1}(q)}
{\sigma_{-1}(q)}-
\left(\frac{\sigma'_{-1}(q)}
{\sigma_{-1}(q)}\right)^{2}
\right),
\end{align*}
where
$\gamma_{0}$ is the Euler's constant and for complex number $l$ and $k\in\mathbb{Z}^{+},$
\[
\sigma_{k}^{(l)}(n)=\sum_{d|n}d^{l}(\log ^{k} d).
\]
Then we can introduce the following lemma.
\begin{Lemma}\label{t3}
Write \[D(x,q) :=\sum_{1\leq n \leq x}\tau(n)\tau(n+q)
=M(x,q)+E(x,q).\]
Then we have
\[
M(x,q)=c_{1}(q)x\log^{2} x+c_{2}(q)x\log x+c_{3}(q)x,
\]
and
$$E(x,q)\ll x^{2/3+\varepsilon}$$
holds uniformly for $1\leq q\leq x^{2/3}.$
\end{Lemma}
\begin{proof}
This is a classical result of Jutila \cite{Ju1}.
\end{proof}
The following lemma has relation to the Fourier coefficients of Hecke-Maass cusp forms.
\begin{Lemma}\label{t5}
Let $\lambda_{\nu}(n)$  be the $n$-th Fourier coefficient of Hecke-Maass cusp forms. Let $\alpha$ be the constant for which $\lambda_{\nu}(n)\ll n^{\alpha}$ holds.
Write \[D(x,q) :=\sum_{1\leq n \leq x}\tau(n)\tau(n+q)
=M(x,q)+E(x,q).\]
Then we have
\[
M(x,q)=c_{1}(q)x\log^{2} x+c_{2}(q)x\log x+c_{3}(q)x,
\]
and
\[
\int_{x}^{2x}E(u,q)^{2}du\ll q^{2\alpha}x^{2+\varepsilon},
\]
provided that
$1\leq q\leq x^{1/(2-2\alpha)}.$
\end{Lemma}
\begin{proof}
See Theorem 6 in \cite{IM}.
\end{proof}
\section{Proof of  theorem \ref{th3}}

We reduce the formula of Theorem \ref{th3}  to exponential sums estimates.   Observe that $[n^{c}]=m$ is equivalent to
\[
-(m+1)^{\gamma}<-n\leq -m^{\gamma},
\]
where $\gamma=1/c.$
Therefore, we have
\begin{align*}
&\sum_{1\leq n\leq x}f([n^{c}])\\
&= \sum_{1\leq m\leq x^{c}}\left([-m^{\gamma}]-[-(m+1)^{\gamma}]
\right)f(m)\\
&=
\sum_{1\leq m\leq x^{c}}\left(-m^{\gamma}+(m+1)^{\gamma}
\right)f(m)
\\&+\sum_{1\leq m\leq x^{c}}\left(\psi(-(m+1)^{\gamma}
)-\psi(-m^{\gamma})\right)f(m)+O(1).
\end{align*}
Then the main term
\[
\sum_{1\leq m\leq x^{c}}\left(-m^{\gamma}+(m+1)^{\gamma}
\right)f(m)
\]
can be evaluated by partial summation.
By Lemma \ref{z3}, for any $H\geq 1$, there exist
functions $\psi^{*}(x)$ and $\delta(x),$ with $\delta(x)$ non-negative, such that
\[
\psi(x)=\psi^{*}(x)+O(\delta(x)),
\]
where
\[
\psi^{*}(x)=\sum_{1\leq h \leq H} a(h)e(hx)
\]
and
\[
\delta(x)=\sum_{|h|\leq H}b(h) e(hx)
\]
with
\[
a(h)\ll h^{-1},\ b(h)\ll H^{-1}.
\]
Then by breaking into dyadic intervals, we have
\begin{align*}
&\sum_{1\leq m\leq x^{c}}\left(\psi (-(m+1)^{\gamma}
)-\psi (-m^{\gamma})\right)f(m)
\\&=
\sum_{1\leq m\leq x^{c}}\left(\psi^{*}(-(m+1)^{\gamma}
)-\psi^{*}(-m^{\gamma})\right)f(m)
\\&+O\left(
\left|\sum_{1\leq m \leq x^{c}}
\left(\delta(-(m+1)^{\gamma})
-\delta(-(m)^{\gamma})\right)\right|\right)
\\& =
O\left(x^{\varepsilon}\left|\sum_{1\leq m \leq x^{c}}\left(\psi^{*}(-(m+1)^{\gamma}
)-\psi^{*}(-m^{\gamma})\right)f(m)\right|
\right.
\\&+\left.x^{\varepsilon}
\left|\sum_{m\sim \mathcal{M}}
\left(\delta(-(m+1)^{\gamma})
-\delta(-(m)^{\gamma})\right)\right|\right)
:=O\left(S_{1}+S_{2}\right),
\end{align*}
where $1\leq \mathcal{M} \leq x^{c}.$
We fix a small constant $\eta>0$ and set
$H= x^{c-1+\eta}.$ Then for $1<c<2$, by  using Lemma \ref{z2}, we obtain
\begin{align*}S_{2}&\ll \frac1H
\sum_{|h|\leq H}x^{\varepsilon}\left|\sum_{n\sim \mathcal{M}}e(hn^{\gamma})\right|\\
&\ll \frac{\mathcal{M}x^{\varepsilon}}{H}
+H^{-1}x^{\varepsilon}\sum_{1\leq h\leq H}\left(h^{1/2}\mathcal{M}^{\gamma/2}
+h^{-1}\mathcal{M}^{1-\gamma}\right)
\ll
x^{1-\varepsilon}.\end{align*}
The remaining task is to prove that
\[
S_{1}\ll x^{1-\eta},
\]
provided that $\eta>0$ is sufficiently small. We write
\[
S_{1}=\sum_{1\leq |h|\leq H}\sum_{1\leq m \leq x^{c}}f(m)a(h)
\phi_{h}(m)e(-hm^{\gamma}),
\]
where $\phi_{h}(y)=
e(h(y^{\gamma}-(y+1)^{\gamma}))-1.$ By using partial summation and the bounds
\[
\phi_{h}(y)\ll hy^{\gamma-1},\ \
\frac{d \phi_{h}(y)}{dy}\ll hy^{\gamma-2},
\]
we have
\[
S_{1} \ll x^{1-c+\varepsilon}\max_{\mathcal{M}'\in
[1,x^{c}]}
\sum_{1\leq h\leq H}\left|\sum_{1\leq m\leq \mathcal{M}'}f(m)e(-hm^{\gamma})\right|.
\]
Hence, it suffices to show that for $\mathcal{M}'\in
[1,x^{c}]$ and $|\varepsilon_{h}|=1$, the following bound
\[
S_{3}:=\sum_{1\leq h\leq H}\varepsilon_{h}\sum_{1\leq m\leq \mathcal{M}'}f(m)e(hm^{\gamma})\ll x^{c-\varepsilon}
\]
holds.
To estimate the above formula, we also need
the Theorem 3 in \cite{RS}.
By the relation $$f(n)=\sum_{n_{1}n_{2}=n}
\tau(n_{1})g(n_{2})$$
and symmetry of divisor function $\tau(n)=\sum_{n_{1}n_{2}=n}1,$
we have
\begin{align*}
S_{3}&\ll H^{\varepsilon}\sum_{1\leq n_{1}\leq \mathcal{M}'}
|g(n_{1})|\sum_{h\sim H_{0}}|\varepsilon_{h}|
\sum_{1 \leq n_{2}\leq \frac{\mathcal{M}'^{1/2}}{n_{1}^{1/2}}}
\left|\sum_{1\leq n_{3}\leq \frac{\mathcal{M}'}
{n_{1} n_{2}}}e(hn_{1}^{\gamma}n_{2}^{\gamma}
n_{3}^{\gamma})\right|,
\\
&\ll (Hx)^{\varepsilon}
\sum_{1\leq n_{1}\leq \mathcal{M}'}
|g(n_{1})|\sum_{h\sim H_{0}}|\varepsilon_{h}|
\sum_{n_{2}\sim N_{2}}
\left|\sum_{n_{3}\sim N_{3}}
e(hn_{1}^{\gamma}n_{2}^{\gamma}n_{3}^{\gamma})
\right|,
\end{align*}
where
$1\leq H_{0}\leq H$, $1\leq N_{2} \leq \frac{\mathcal{M'}^{1/2}}{n_{1}^{1/2}}$, $1\leq N_{3} \leq \frac{\mathcal{M' }}{n_{1}}$ and $1\leq N_{2}N_{3} \leq \frac{\mathcal{M'} }{n_{1} }$.
Then by  Lemma \ref{rs1}, we have
\begin{align*}
S_{3}&\ll
(Hx)^{\varepsilon}
\sum_{1\leq n_{1}\leq \mathcal{M}'}g(n_{1})
\left(\frac{H\mathcal{M'}^{\gamma/4+5/8}}
{n_{1}^{5/8}}+ \frac{H\mathcal{M'}^{3/4}}
{n_{1}^{3/4}}+\frac{\mathcal{M'}^{1-\gamma}}
{n_{1}^{1-\gamma}}
\right)\\
&\ll
(Hx)^{\varepsilon}
\left( H\mathcal{M'}^{\gamma/4+5/8} + H\mathcal{M'}^{3/4} + \mathcal{M'}^{5/8}
\right),
\end{align*}
provided that
\[
\sum_{1\leq n\leq x}|g(n)|\ll x^{5/8+\varepsilon}.
\]
Recall that $H=x^{c-1+\eta}.$ Then we have
\begin{align*}
S_{3}
&\ll
(Hx)^{\varepsilon}
\left( \mathcal{M'}^{1-\gamma+\eta}\mathcal{M'}^{\gamma/4+5/8} + \mathcal{M'}^{1-\gamma+\eta}\mathcal{M'}^{3/4} + \mathcal{M'}^{5/8}
\right)\\
&\ll x^{c-\varepsilon},
\end{align*}
provided  that $c\in (1,6/5).$
This completes the proof.

\section{Proof of Theorem \ref{th5}}
Starting similar as in the proof of Theorem 1.1,
we can show with minimal changes that
\begin{align*}
&\sum_{\substack{1\leq n\leq x\\ [n^{c}]\equiv a(\mod d)}}f([n^{c}])\\
&= \sum_{\substack{1\leq m\leq x^{c}\\ m\equiv a(\mod d)}} \left([-m^{\gamma}]-[-(m+1)^{\gamma}]
\right)f(m)\\
&+
O\left(T_{1}+T_{2}\right)+O(1),
\end{align*}
where
\[T_{1}:=\left|x^{\varepsilon}\sum_{
\substack{m\sim \mathcal{M}\\ m\equiv a(\mod d)}} \left(\psi^{*}(-(m+1)^{\gamma}
)-\psi^{*}(-m^{\gamma})\right)f(m)
\right|,
\]
\[T_{2}:= \left|x^{\varepsilon}
\sum_{m\sim \mathcal{M}}
\left(\delta(-(m+1)^{\gamma})
-\delta(-(m)^{\gamma})\right)\right|\]
and $1\leq \mathcal{M} \leq x^{c}.$
Then the main term
\[
\sum_{\substack{1\leq m\leq x^{c}\\ m\equiv a(\mod d)}} \left(-m^{\gamma}+(m+1)^{\gamma}
\right)f(m)
\]
can be evaluated by partial summation.

We fix a small constant $\eta>0$ and set
$H= x^{c-1+\eta}.$ Then for $1<c<2$, using Lemma \ref{z2}, we obtain
\begin{align*}T_{2}&\ll \frac1H
\sum_{|h|\leq H}x^{\varepsilon}\left|\sum_{n\sim \mathcal{M}}e(hn^{\gamma})\right|\\
&\ll \frac{\mathcal{M}x^{\varepsilon}}{H}
+H^{-1}x^{\varepsilon}\sum_{1\leq h\leq H}\left(h^{1/2}\mathcal{M}^{\gamma/2}
+h^{-1}\mathcal{M}^{1-\gamma}\right)
\ll
x^{1-\varepsilon}.\end{align*}
The remaining task is to prove that
\[
T_{1}\ll x^{1-\eta},
\]
provided that $\eta>0$ is sufficiently small. We write
\[
T_{1}=\sum_{1\leq |h|\leq H}\sum_{\substack{m\sim\mathcal{M}\\ m\equiv a(\mod d)}} f(m)a(h)
\phi_{h}(m)e(-hm^{\gamma}),
\]
where $\phi_{h}(y)=
e(h(y^{\gamma}-(y+1)^{\gamma}))-1.$ By using partial summation and the bounds
\[
\phi_{h}(y)\ll hy^{\gamma-1},\ \
\frac{d \phi_{h}(y)}{dx}\ll hy^{\gamma-2},
\]
we have
\[
T_{1} \ll x^{1-c+\varepsilon}\max_{
\mathcal{M}'\in
[\mathcal{M},2\mathcal{M}]}
\sum_{1\leq h\leq H}\left|\sum_{
\substack{\mathcal{M}\leq m\leq \mathcal{M}'\\ m\equiv a(\mod d)}
}f(m)e(-hm^{\gamma})\right|.
\]
Hence, it suffices to show that for $\mathcal{M}'\in
[\mathcal{M},2\mathcal{M}]$ and $|\varepsilon_{h}|=1$, the following bound
\[
T_{3}:=\sum_{1\leq h\leq H}\varepsilon_{h}\sum_{
\substack{\mathcal{M}\leq m\leq \mathcal{M}'\\ m\equiv a(\mod d)}
}f(m)e(hm^{\gamma})\ll x^{c-\varepsilon}
\]
or
\[
T_{3}=\frac{1}{d}\sum_{1\leq t\leq d}\sum_{1\leq h\leq H}\varepsilon_{h}\sum_{
\substack{\mathcal{M}\leq m\leq \mathcal{M}'}
}f(m)e\left(hm^{\gamma}
+\frac{(m-a)t}{d}\right)\ll x^{c-\varepsilon}
\]
holds.

By the relation $$f(n)=\sum_{n_{1}n_{2}=n}
\tau(n_{1})g(n_{2})$$
and symmetry of divisor function $\tau(n)=\sum_{n_{1}n_{2}=n}1,$
we have
\begin{align}\label{a0}
\begin{split}
T_{3}&\ll \frac{1}{d}\sum_{1\leq t\leq d}H^{\varepsilon}\sum_{1\leq n_{1}\leq \mathcal{M}'}
|g(n_{1})|\sum_{h\sim H_{0}}|\varepsilon_{h}|
\left|\sum_{\frac{\mathcal{M}}
{n_{1}} \leq n_{2}\leq \frac{\mathcal{M}'}{n_{1}}}
\tau(n_{2})e\left(hn_{1}^{\gamma}n_{2}^{\gamma}
+\frac{n_{1}n_{2}t}{d}\right)\right|,
\\
&\ll \frac{1}{d}\sum_{1\leq t\leq d}(Hx)^{\varepsilon}
\sum_{  n_{1}\sim N_{1}}
|g(n_{1})|\sum_{h\sim H_{0}}|\varepsilon_{h}|
\left|\sum_{n_{2}\sim N_{2}}
\tau(n_{2})e\left(hn_{1}^{\gamma}n_{2}^{\gamma}
+\frac{n_{1}n_{2}t}{d}\right)
\right|,
\end{split}
\end{align}
where
$1\leq H_{0}\leq H$, $N_{1}N_{2}\ll \mathcal{M}$.
By Lemma \ref{t1} and partial summation, we have
\begin{align}\label{a1}
\begin{split}
&\left|\sum_{n_{2}\sim N_{2}}
\tau(n_{2})e\left(hn_{1}^{\gamma}n_{2}^{\gamma}
+\frac{n_{1}n_{2}t}{d}\right)\right|^{2}
\\&\ll \frac{N_{2}^{2}}{Q}
+\frac{N_{2}}{Q}\sum_{1\leq q\leq Q}\left(1-\frac{q}{Q}\right)
\left|\sum_{n_{2}\sim N_{2}}
\tau(n_{2})\tau(n_{2}+q) e\left(hn_{1}^{\gamma}n_{2}^{\gamma}-hn_{1}^{\gamma}
(n_{2}+q)^{\gamma}\right)
\right|\\
&\ll
\frac{N_{2}^{2}}{Q}+
\sum_{1\leq q\leq Q}\left(1-\frac{q}{Q}\right)
\\&\times\left|\int_{N_{2}}^{2N_{2}} e\left(hn_{1}^{\gamma}u^{\gamma}
-hn_{1}^{\gamma}
(u+q)^{\gamma}\right)d\left(
\sum_{1\leq n_{2}\leq u}
\tau(n_{2})\tau(n_{2}+q) \right)
\right|,
\end{split}
\end{align}
where we have ignored the negative values of $q$, which may be treated in the same way.
By Lemma \ref{t3}, we have
\begin{align}\label{a2}
\begin{split}
&\int_{N_{2}}^{2N_{2}} e\left(hn_{1}^{\gamma}u^{\gamma}
-hn_{1}^{\gamma}
(u+q)^{\gamma}\right)d\left(
\sum_{1\leq n_{2}\leq u}
\tau(n_{2})\tau(n_{2}+q)\right)\\
&=\sum_{l=0}^{2}c_{l}(q) \int_{N_{2}}^{2N_{2}} e\left(hn_{1}^{\gamma}u^{\gamma}
-hn_{1}^{\gamma}
(u+q)^{\gamma}\right)(\log u)^{l}du
\\&+\sum_{l=1}^{2}c_{l}(q)l \int_{N_{2}}^{2N_{2}} e\left(hn_{1}^{\gamma}u^{\gamma}
-hn_{1}^{\gamma}
(u+q)^{\gamma}\right)(\log u)^{l-1}du
\\&+O\left(E(N_{2},q)\right)
+O\left(E(2N_{2},q)\right)\\
\\&+O\left(hn_{1}^{\gamma}q
N_{2}^{\gamma-2}\int_{N_{2}}^{2N_{2}}
\left|E(u,q)\right|du.
\right)
\end{split}
\end{align}
By Lemma \ref{t4}, we have
\begin{align}\label{a3}
\begin{split}
&\sum_{l=0}^{2}c_{l}(q) \int_{N_{2}}^{2N_{2}} e\left(hn_{1}^{\gamma}u^{\gamma}
-hn_{1}^{\gamma}
(u+q)^{\gamma}\right)(\log u)^{l}du
\\&
+\sum_{l=1}^{2}c_{l}(q)l \int_{N_{2}}^{2N_{2}} e\left(hn_{1}^{\gamma}u^{\gamma}
-hn_{1}^{\gamma}
(u+q)^{\gamma}\right)(\log u)^{l-1}du
\ll \frac{N_{2}}{hqn_{1}
^{\gamma}N_{2}^{\gamma}}.
\end{split}
\end{align}
By Lemma \ref{t3}, we have
\begin{align}\label{a4}
\begin{split}
E(N_{2},q)
\ll N_{2}^{2/3+\varepsilon}\end{split}
\end{align}
and
\begin{align}\label{a5}
\begin{split}E(2N_{2},q)\ll N_{2}^{2/3+\varepsilon},
\end{split}
\end{align}
provided that $2\leq q\leq N_{2}^{2/3}.$
By Lemma \ref{t5} and Cauchy's inequality,
we have
\begin{align}\label{a6}
\int_{N_{2}}^{2N_{2}}
\left|E(u,q)\right|du\ll q^{\alpha}N_{2}^{3/2},
\end{align}
provided that $1\leq q \leq N_{2}^{\frac{1}{2-2\alpha}}.$
Substituting (\ref{a3}),(\ref{a4}),(\ref{a5}) and (\ref{a6}) into (\ref{a2}) yields that
\begin{align}\label{a7}
\begin{split}
&\int_{N_{2}}^{2N_{2}} e\left(hn_{1}^{\gamma}u^{\gamma}
-hn_{1}^{\gamma}
(u+q)^{\gamma}\right)d\left(
\sum_{1\leq n_{2}\leq u}
\tau(n_{2})\tau(n_{2}+q)\right)\\
&\ll \frac{N_{2}}{hqn_{1}
^{\gamma}N_{2}^{\gamma}}
+N_{2}^{2/3+\varepsilon}
+hn_{1}^{\gamma}q
N_{2}^{\gamma-2}q^{\alpha}N_{2}^{3/2}
\end{split}
\end{align}
Substituting (\ref{a7})  into (\ref{a1}) yields that
\[\left|\sum_{n_{2}\sim N_{2}}
\tau(n_{2})e\left(hn_{1}^{\gamma}n_{2}^{\gamma}
+\frac{n_{1}n_{2}t}{d}\right)\right|^{2}
\ll \frac{N_{2}^{2+\varepsilon}}{Q}
+N_{2}^{5/3+\varepsilon}+hn_{1}^{\gamma}
N_{2}^{\gamma+1/2}Q^{1+\alpha}.\]
By Lemma \ref{t2}, choosing $Q\in(0,N^{1/2-\varepsilon}),$ we have
\begin{align*}
&\frac{1}{d}\sum_{1\leq t\leq d}H^{\varepsilon}\sum_{1\leq n_{1}\leq \mathcal{M}'}
|g(n_{1})|\sum_{h\sim H_{0}}|\varepsilon_{h}|\left|\sum_{n_{2}\sim N_{2}}
\tau(n_{2})e\left(hn_{1}^{\gamma}n_{2}^{\gamma}
+\frac{n_{1}n_{2}t}{d}\right)\right|
\\&\ll (Hx)^{\varepsilon}\sum_{1\leq n_{1}\leq \mathcal{M}'}g(n_{1})\left(
HN_{2}^{3/4+\varepsilon}
+H^{1+\frac{1}{2(2+\alpha)}}
(n_{1}N_{2})^\frac{\gamma}{2(2+\alpha)}
N_{2}^{\frac{1}{4(2+\alpha)}
+\frac{1+\alpha}{2+\alpha}}
+HN_{2}^{5/6}\right)
\\
&\ll
(Hx)^{\varepsilon}\left(
H\mathcal{M}^{5/6}
+H^{1+\frac{1}{2(2+\alpha)}}
\mathcal{M}^\frac{\gamma}{2(2+\alpha)}
\mathcal{M}^{\frac{5+4\alpha}{4(2+\alpha)}}
\right),
\end{align*}
provided that
\[
\sum_{1\leq n\leq x}|g(n)|
\ll x^{\frac{5+4\alpha}{8+4\alpha}
+\varepsilon}.
\]
Then we have
\begin{align*}
T_{3}\ll
 x^{c-\varepsilon},
\end{align*}
provided  that $c\in \left(1,\frac{8+4\alpha}
{7
+4\alpha}\right)$ and
\[
\sum_{1\leq n\leq x}|g(n)|
\ll x^{\frac{5+4\alpha}{8+4\alpha}
+\varepsilon}.
\]
This completes the proof.

\bigskip

\vspace{0.5cm}
$\mathbf{Acknowledgements}$
The author thanks the anonymous referees for helpful comments that
improved the paper.
This work was supported by Natural Science Foundation of China (Grant No. 12401004) and Natural Science Foundation of Henan Province (Grant No. 242300421676).

\address{Wei Zhang\\ School of Mathematics\\
               Henan University\\
               Kaifeng  475004, Henan, China}
\email{zhangweimath@126.com}
\end{document}